\documentclass[a4paper,11pt]{amsart}
\usepackage{hyperref}

\usepackage[usenames,dvipsnames,svgnames,table]{xcolor}
\usepackage[all]{xy}
\usepackage{enumitem}
\usepackage[normalem]{ulem}
\newcommand{\tsk}[1]{\textcolor{YellowOrange}}

\usepackage{tikz}
\usepackage{tikz-cd}
\usepackage{graphicx}
\usepackage{pgf,tikz}
\usetikzlibrary{arrows}
\usepackage[left=2cm,right=2cm, top=3cm, bottom=3cm]{geometry}

\makeatletter
\def\@endtheorem{\endtrivlist}
\makeatother

\usepackage[active]{srcltx}
\usepackage{amsmath}
\usepackage{amssymb}
\usepackage{amscd}
\usepackage{amsthm}
\usepackage{mathrsfs}  
\usepackage{mathtools, nccmath}

\usepackage{nicefrac}

\newcommand{\Pic}{\operatorname{Pic}}

\newtheorem{teo}{Theorem}[section]

{}

\newtheorem{proposition}[teo]{Proposition}
\newtheorem{cor}[teo]{Corollary}

\newtheorem{lemma}[teo]{Lemma}

\theoremstyle{definition}
\newtheorem{mainconstruction}[teo]{Main construction}
\newtheorem{notations}[teo]{Notations}
\theoremstyle{definition}

\newtheorem{remark}[teo]{Remark}

\newtheoremstyle{dico}
{\baselineskip}   
{\topsep}   
{}  
{0pt}       
{} 
{.}         
{5pt plus 1pt minus 1pt} 
{}          
\theoremstyle{dico}

\numberwithin{equation}{section}

\newcounter{example}[subsection]
\newenvironment{example}[1][]
{\refstepcounter{example}\smallskip\par\noindent\textbf{Example~\theexample.}\nopagebreak\par\noindent}{\medskip}

\newcommand{\ra}{\rightarrow}

\newcommand{\restr}[1]          {\vert_{#1}}

\renewcommand{\phi}{\varphi}

\def\dg#1{\Delta_{#1}}

\newcommand{\mihi}[1]{}

\begin{document}

\pagestyle{myheadings}

\title{Gaussian maps for singular curves on Enriques surfaces}

\author{Dario Faro}

\address{Universit\`a degli Studi di Pavia, Dip. di Matematica, Via Ferrata 5, 27100 Pavia, Italy.}
\email{d.faro1@campus.unimib.it}

\begin{abstract}
We give obstructions - in terms of Gaussian maps -  for a marked Prym curve $(C,\alpha,T_d)$ to admit a singular model lying on an Enriques surface with only one $d$-ordinary point singularity and in such a way that  $T_d$ corresponds to the divisor over the singular point.
\end{abstract}
\maketitle
\setcounter{section}{-1}

\section{Introduction}
The article deals with the problem of finding obstructions - in terms of Gaussian maps - for a Prym curve $(C,\alpha)$ to admit a singular model (with prescribed singularities) in a polarized Enriques surface $(S,H)$. Let us briefly introduce the setting. Let $X$ be a smooth complex projective variety, and let $L$ and $M$ be two invertible sheaves on $X$. Denote by $R(L,M)$ the kernel of the multiplication map $\Phi^0_{L,M}: H^0(X,L) \otimes H^0(X,M) \rightarrow H^0(X,L \otimes M)$. The first Gaussian map associated with $L,M$ is the map
$$
\Phi^1_{X,L,M}:R(L,M) \rightarrow H^0(X, \Omega^1_X \otimes L \otimes M)
$$
locally defined as $\Phi^1_{X,L,M}(s \otimes t)=sdt-tds$. If $L=M$ one usually writes $\Phi^1_{X,L}$ and since it vanishes on symmetric tensors, one usually considers its restriction to $\bigwedge^2L$.
Gaussian maps were introduced by Wahl who showed, in \cite{wahl jac}, that if $(S',H')$  is a polarized $K3$ surface and and $C' \in |H'|$ is a curve then  the Gaussian map (also called Wahl map) $\Phi^1_{\omega_C}$ is not surjective (see also \cite{bm} for a different proof). 
 On the other hand Ciliberto, Harris and Miranda proved in \cite{cilharmir} that the Wahl map is surjective as soon as it is numerically possible, i.e. for $g \geq 10, g \neq 11$. The converse also holds. In \cite{arba} Arbarello, Bruno and Sernesi proved that a Brill-Noether-Petri general curve with non-surjective Gaussian map lies in a $K3$ surface or on a limit thereof. See also \cite{cilds}  for a result relating the corank of the Gaussian map and $r$-extendibility. Analogous problems for Enriques surfaces have also been studied  by some authors. Indeed, let $(S,H)$ be  a polarized Enriques surface and let  $C \in |H|$ be a smooth curve and $\alpha:=K_{|_{C}}$. In \cite{ballcil} it is proven that the Gaussian map $\Phi_{\omega_C,\omega_C \otimes \alpha}$ is not surjective, whereas in \cite{cv} it is shown that for the general prym curve $(C,\alpha)$ of genus   $g \geq 12$, $g \neq 13, 19$ the map  is surjective. Gaussian maps  have been studied and used by many authors, either in relation to extendibility questions, we mention e.g.  \cite{ballcil}, \cite{balfonta}, \cite{bfr} - \cite{farspel},  \cite{ccm} - \cite{cf2} - \cite{ortiz}, \cite{cildedieu}, \cite{dufmir}, \cite{knutre} -\cite{knudue}, \cite{cfp} (see also \cite{lopezexten} for a complete  survey),  or in relation to the second fundamental form of Torelli-type immersions, e.g. \cite{cf3}, 
\cite{cf4},  \cite{colfre}, \cite{colpirto}. \cite{fred}. 
The result of Wahl was generalized by some authors, e.g. Zak-L'Vovsky, who proved the following theorem, that we are going to use.
\begin{teo}[\cite{lvovsky}]
\label{L'vosky}
Let $C$ be a smooth curve of genus $g>0$ and let $A$ be a very ample line bundle on $C$,  embedding $C$ in $\mathbb{P}^n$ for $n \geq 3$. If $C \subset \mathbb{P}^n$ is scheme-theoretically a hyperplane section of a smooth surface $X \subset \mathbb{P}^{n+1} $, then the Gaussian map $\Phi_{\omega_C,A}$ is not surjective.
\end{teo}
Similar questions for singular curves on $K3$ surfaces are discussed and solved by Kemeny in \cite{kem1}. In the article the author asks whether one can give an obstruction in terms of suitable gaussian maps for a curve to have a nodal model lying on a $K3$ surface. Following the author notations, denote by $\Bar{\mathcal{M}}_{h,2l}$ the moduli stack of smooth curve of genus $g$ with $2l$ marked point and by $\widetilde{\mathcal{M}}_{h,2l}=\Bar{\mathcal{M}}_{h,2l}/S_{2n}$ the stack of curves with unordered marking. Let $h,l$ be two positive integers and $[C,T] \in \widetilde{\mathcal{M}}_{h,2l}$. The author introduces the marked Wahl map:
\begin{equation}
 W_{C,T}: \bigwedge^2H^0(C,K_C(-T)) \rightarrow H^0(C,K_C^3(-2T)).
\end{equation}
Then the following theorems are proven.
\begin{teo}[\cite{kem1}]
Fix any integer integer $l \in \mathbb{Z}$. Then there exists infinitely many integers $h(l)$, such that the general marked $[(C,T)] \in \widetilde{\mathcal{M}}_{h(l),2l}$ has surjective marked Wahl map.
\end{teo}
Now denote by  $\mathcal{V}^n_{g,k}$ the stack parametrizing morphisms $[(f:C \rightarrow X,L)]$ where $(X,L)$ is a polarized $K3$ surface with  $L^2=2g-2$, $C$ is a smooth connected curve of genus $p(g,k)-n$ with $p(g,k):=k^2(g-1)+1$, $f$ is birational onto its image and $f(C) \in |kL|$ is nodal.
\begin{teo}[\cite{kem1}]
Assume $g-n \geq 13$ for $k=1$ or $g \geq 8$ for $k >1$, and let $n \leq \frac{p(g,k)-2}{5}$. Then there is an irreducible component $I^0 \subseteq \mathcal{V}^n_{g,k}$ such that for a general $[(f:C \rightarrow X,L)] \in I^0$ the marked Wahl map $W_{C,T}$ is non surjective, where $T \subseteq C$ is the divisor over the nodes of $f(C)$.
\end{teo}
The same marked Wahl maps have been studied by Fontanari and Sernesi in \cite{fs1}, where  they proved, using very different methods from \cite{kem1}, the following theorem. 
\begin{teo}[\cite{fs1}]
\label{sernesifontanari}
Fix an integer $g \geq 9$ Let $(S,H)$ be a polarized $K3$ surface with  $Pic(S)=\mathbb{Z}H$ and  $H^2=2g-2$. Let $C$ be a smooth curve of genus $g-1$ endowed with a morphism $f: C \rightarrow S$ birational onto its image and such that $f(C) \in |H|$. If $T=P+Q \subseteq C$ is the divisor over the singular point of $f(C)$, then the Gaussian map $\Phi_{\omega_C - T,\omega_C-T}$ is not surjective. 
\end{teo}
Our paper deals with a similar problem for singular curves on Enriques surfaces. Let $(S,H)$ be a polarized Enriques surface and $C$ a smooth curve having a morphism $f: C \rightarrow S$ birational onto its image and  such that $f(C) \in |H|$ has exactly one ordinary $d$-point or a cusp, in case $d=2$. Denote by $T_d$ the divisor over the singular point and set $\alpha=f^*K_S$. Then $(C, \alpha, T_d)$ is a marked Prym curve. We investigate the behaviour of the following mixed Gaussian-Prym maps: 
\begin{equation}   
\label{prima}
\Phi^1_{C,\omega_C-T_d,\omega_C-T_d + \alpha}:R(\omega_C-T_d,\omega_C-T_d + \alpha) \rightarrow H^0(X, \omega^3_C  \otimes \alpha (-2T_d))
\end{equation}
and
\begin{equation}  
\label{seconda}
\Phi^1_{C,\omega_C,\omega_C-T_d + \alpha}:R(\omega_C,\omega_C-T_d + \alpha) \rightarrow H^0(X, \omega^3_C  \otimes \alpha (-T_d)).
\end{equation}
More precisely we have the following.
\begin{teo}
\label{teononsurjfgh}
Let $(S,H)$ be a polarized Enriques surface with $H^2=2g-2$ and let $d \geq 2$. Suppose that either
\begin{itemize}

\item[(i)] $S$ is a very general Enriques surface and $\phi(H) \geq \sqrt{2}(d+2)$, or
   
 \item[(ii)] $S$ in unnodal and $\phi(H) \geq2(d+1)$.  
\end{itemize}
Set $g'=g-\binom{d}{2}$ and let $C$ be a smooth curve of genus $g'$ having a birational morphism $f: C \ra S$ onto its image and  such that $f(C) \in |H|$, $f(C)$ has exactly one ordinary $d-$point or a cusp, in case $d=2$. Set  $\alpha=f^*K_{S_{|_{C}}}$ and let  $T_d=p_1+...+p_d$ be the divisor over the singular point. Then the Gaussian maps $\Phi_{\omega_C,\omega_C-T_d+\alpha}$ and  $\Phi_{\omega_C-T_{d},\omega_C-T_{d}+\alpha}$ are not surjective.
\end{teo}
Here, with ``general", we mean in a non empty Zariski-open subset of the moduli, with ``very general" we mean outside a countable union of proper Zariski-closed subset.  The proof is along  the same lines of Theorem \ref{sernesifontanari}

On the contrary, when one considers a general marked Prym curve the aforementioned maps are  ``tendendially" surjective. 
 Indeed, let $S$ be the following set:
\begin{equation}
S:=\{(g_1,d_1,d_2):  g_1 \geq 3, d_1 \geq 5, d_2 \geq 4, d_2(g_1-2)>d_1 \geq g_1+5, d_1 >d_2  \},
\end{equation}
and denote by $R_{g,d}$ the moduli space of $d-$marked Prym curve. We prove the following.
\begin{teo}
\label{propsurjgeneral23}
Let $(g_1,d_1,d_2)$ be in $S$ (\ref{equadhfd}), and $g=(g_1-2)d_2+d_1(d_2-1)+1$.  Let $d$ be an integer such that $ 2 \leq d \leq d_2$. If  $(C,\alpha,T_d )$ is a general element in $R_{g,d}$, then the Gaussian maps
$$
    \Phi_{C, \omega_C - T_d,\omega_C-T_d + \alpha}
    $$
    and
    $$
    \Phi_{C,\omega_C, \omega_C - T_d + \alpha}
    $$
    are  surjective.
\end{teo}
In  case $d=2,3$ or $d=4$ (see example \ref{exampledsds} ) we obtain the surjectivity for all genera greater  than or equal to $41$. More generally, for every $m$ we obtain infinitely many genera for which the marked Gaussian maps (we are considering) are surjective. We expect our result far from being sharp (see remark \ref{remarksbounds}).

We briefly explain how the paper is organized. In section \ref{section1} we recall the definition of the Gaussian maps and prove (proposition \ref{propositionfundamentalsurjection})  a slight modification of \cite{fs1}, Theorem $8$ (see also Theorem $9$). This is a result relating cokernels of Gaussian maps in different embeddings. 
In section \ref{section2} we prove  theorem \ref{teononsurjfgh}, following the same strategy of the proof of theorem \ref{sernesifontanari} (\cite{fs1}). In particular, we show some results regarding very ampleness of line bundles on the blow-up of Enriques surfaces.  In \ref{section3} we prove the surjectivity of the marked Prym-gaussian maps for a certain class of $d-$ marked Prym curves living in the product of two curves. In section \ref{section4} we give a lower bound for the gonality of curves living in the product of a curve with $\mathbb{P}^1$  (see proposition \ref{goncurvesprod} )  and we prove a lemma about very ample line bundles on curves. We will use them  in the proof of theorem \ref{propsurjgeneral23}. In section \ref{section5} we finally prove theorem \ref{propsurjgeneral23}.

\section*{Acknowledgments}
I deeply thank  Paola Frediani and Andreas Leopold Knutsen for several discussions about the paper and for their fundamental remarks and suggestions. I also want to thank Andrea Bruno,  Margherita Lelli Chiesa, Angelo Felice Lopez, and Edoardo Sernesi for a conversation about the paper and their suggestions.
\section{Cokernels of Wahl maps}
\label{section1}
We briefly recall the definition of the Gaussian maps, and their different interpretations, which will be used in the sequel. See for example \cite{wahl1} or \cite{wahl2} for the details. 
Let $X$ be a smooth complex algebraic variety. Let $L$ and $M$ be two line bundles on $X$. Let $q_i: X \times X \rightarrow X$, $i=1,2$ be the two projections. Consider the short exact sequence defining the square of the diagonal $\dg{X \times X}$ and tensor it with $q_1^*L \otimes q_2^*M$
\begin{equation}
0 \rightarrow I^2_{\dg{X \times X}} \otimes  q_1^*L \otimes q_2^*M  \rightarrow  I_{\dg{X \times X}} \otimes q_1^*L \otimes q_2^*M \rightarrow I_{\dg{X \times X}}/ I^2_{\dg{X \times X}} \otimes q_1^*L \otimes q_2^*M \rightarrow 0. 
\end{equation}
 The first Gaussian  map  associated with $L$ and $M$ is defined as the map induced at the level of global sections:
 $$
\Phi_{L,M}:H^0(X \times X,   I_{\dg{X \times X}} \otimes q_1^*L \otimes q_2^*M) \rightarrow H^0(X \times X, I_{\dg{X \times X}}/I^2_{\dg{X \times X}}  \otimes q_1^*L \otimes q_2^*M).
$$
Now let $\Phi^0_{L,M}: H^0(X, L) \otimes H^0(X,M) \rightarrow H^0(X, L \otimes M)$ be the multiplication map and denote by $R(L,M)$ its kernel.  Using standard identifications,  $\Phi_{L,M}$ can be thought as a map 
\begin{equation*}
    R(L,M) \rightarrow  H^0(X, \Omega^1_X \otimes L \otimes M).
\end{equation*}
 If $\alpha=\sum l_i \otimes m_i \in \text{Ker}( \phi_{L,M})$, $l_i=f_iS$, $m_i=s_iT$, where $S$ and $T$ are two local generators of $L$ and $M$, respectively, it is locally given by 
$\Phi_{L,M}(\alpha)=\sum (f_idg_i-g_idf_i) S \otimes T$.
Now we recall a different interpretation of the Gaussian maps in the hypothesis of  $L$ being  very ample and  giving an embedding in $\phi_L: X \xhookrightarrow{} \mathbb{P}^r$. 

Denote by $M_L$ the kernel of the evaluation map of sections of $L$, i.e.:
$$
0 \rightarrow M_L \rightarrow H^0(C,L) \otimes \mathcal{O}_C \rightarrow L \rightarrow 0. 
$$
Then $\phi_L^*\Omega^1_{\mathbb{P}^r}(1)=\Omega^1_{\mathbb{P}^r}(1)_{|_{X}} \simeq M_L$. Consider indeed the Euler sequence
\begin{equation*}
    0 \rightarrow \Omega_{\mathbb{P}^r} \rightarrow \mathcal{O}_{\mathbb{P}^r}(-1)^{r+1} \rightarrow \mathcal{O}_{\mathbb{P}^r}\rightarrow 0,
\end{equation*}
and tensor it with $\mathcal{O}_{P^r}(1)$:
\begin{equation*}
    0 \rightarrow \Omega^1_{\mathbb{P}^r}(1) \rightarrow H^0(\mathbb{P}^r,\mathcal{O}_{\mathbb{P}^r}(1))  \otimes \mathcal{O}_{\mathbb{P}^r} \rightarrow \mathcal{O}_{\mathbb{P}^r}(1)\rightarrow 0.
\end{equation*}
Pullbacking it  by $\phi_L$ we obtain
\begin{equation*}
0 \rightarrow \phi_L^*\Omega_{\mathbb{P}^r}(1) \rightarrow H^0(C,L) \otimes \mathcal{O}_C \rightarrow L \rightarrow 0,
\end{equation*}
and so we conclude.
Now consider a twist by $L \otimes M$ of  the conormal exact sequence:
\begin{equation}
\label{eqconorm}
0 \rightarrow N^{\vee}_{X / \mathbb{P}^r} \otimes L \otimes M \rightarrow M_L \otimes M \rightarrow \Omega_X \otimes L  \otimes M\rightarrow 0,  
\end{equation}
One can show that under the aforementioned identification,  
\begin{equation*}
\Phi_{L,M}: H^0(X,M_L \otimes M) \rightarrow H^0(X,\Omega_X \otimes L  \otimes M),
\end{equation*}
i.e. $\Phi_{L,M}$ is the map  induced at the level of global section in \ref{eqconorm}. 
Now we recall a very useful construction of Lazarsfeld.
\begin{proposition}[Lemma 1.4.1, \cite{laz89})]
\label{proplazz}
Let $p_1,...,p_n \in X$ be distinct points such that $L(-\sum_{i=1}^np_i)$ is generated by global sections, and $h^1(L(-\sum_{i=1}^np_i))=h^1(L)$. Then one has an exact sequence:
\begin{equation}
\label{lazcono}
0 \rightarrow M_{L(-\sum_{i=1}^np_i)} \rightarrow M_{L}  \rightarrow \bigoplus_{i=1}^n\mathcal{O}_{C}(-p_i) \rightarrow 0.
\end{equation}
\end{proposition}
We now observe that a slight modification of \cite{fs1},Theorem $8$, gives the following result which relates cokernels of gaussian maps in different embeddings. In the following $X=C$ will be a smooth complex algebraic curve.  
\begin{proposition} \label{propositionfundamentalsurjection}
Let $C$ be a smooth  complex projective algebraic curve. Let $T_n=p_1+...+p_n$ be an effective divisor of degree $n$ on $C$ with $p_i \neq p_j$ for $i \neq j$. Let $L$ and $M$ be two very ample line bundles such that  $L-T_n$ is very ample and $h^1(L)=h^1(L-T_n)$. Then there exists a surjection between the cokernels of the gaussian maps:
\begin{equation*}
 \operatorname{coker}(\Phi_{L-T_n,M}) \rightarrow \operatorname{coker}(\Phi_{L,M}).   
\end{equation*}
\end{proposition}
The proof follows the same steps of \cite{fs1},Theorem $8$. We present it for completeness. 
\begin{proof}
Consider the following commutative diagram.
%
%
%
The first two rows are \eqref{eqconorm} for the line bundles $L$ and $L-T_n$, the second column is \eqref{lazcono} twisted by $M$, and  the third column is just the restriction modulo the identification $\omega_C L \otimes \mathcal{O}_T  \simeq \bigoplus_{i=1}^n\mathcal{O}_{p_{i}}(-p_i)$, and then twisted by $M$. 
\begin{center}
\begin{tikzcd}
& 0 \arrow[d] & 0 \arrow[d] & 0 \arrow[d]\\
    
0 \arrow[r] &  N^{ \vee}_{C/\mathbb{P}^{r-n}} \otimes L(-T_n)\otimes M \arrow[r,] \arrow[d,]
    & M_{L(-T_n)} \otimes M\arrow[r,] \arrow[d,] & \omega_C \otimes L(-T) \otimes M \arrow[r]\arrow[d] & 0 \\
    
0  \arrow[r]    & N^{ \vee}_{C/\mathbb{P}^{r}} \otimes L \otimes M  \arrow[d] \arrow[r,]
&  M_{L}  \otimes M \arrow[d] \arrow[r] & \omega_C \otimes L \otimes M \arrow[d] \arrow[r] & 0 \\

0  \arrow[r]    &  \bigoplus_{i=1}^n M(-2p_i)  \arrow[d] \arrow[r,"g"]
&  \bigoplus_{i=1}^n M(-p_i) \arrow[d] \arrow[r] & \bigoplus_{i=1}^n M_{|_{p_{i}}}(-p_i) \arrow[d] \arrow[r] & 0\\
& 0 & 0 & 0 & 
\end{tikzcd}
\end{center}
Passing to cohomology we obtain 
\begin{center}
\begin{tikzcd}
& 0 \arrow[d] & 0 \arrow[d] & 0 \arrow[d]\\
    
0 \arrow[r] &  \operatorname{coker}(\Phi_{L(-T_{n})},M)  \arrow[r,] \arrow[d,]
    & H^1(N^{ \vee}_{C/\mathbb{P}^{r-n}} \otimes L \otimes(-T_n) \otimes M)  \arrow[r,] \arrow[d,] & H^1(M_{L(-T_{n})}) \otimes M ) \arrow[r]\arrow[d] & 0 \\
    
0  \arrow[r] &  \operatorname{coker}(\Phi_{L,M})  \arrow[r,] \arrow[d,]
    & H^1(N^{ \vee}_{C/\mathbb{P}^{r}} \otimes L \otimes M)  \arrow[r,] \arrow[d,] & H^1( M_L \otimes M) \arrow[r]\arrow[d] & 0 \\

0  \arrow[r,]    & \operatorname{ker(H^1(g))}  \arrow[r,]
&  H^1( \bigoplus M(-2p_i))  \arrow[r,"H^1(g)"] & H^1( \bigoplus M(-p_i))  \arrow[r] & 0\\
\end{tikzcd}
\end{center}
Being $M$ very ample we have that  $h^1(\bigoplus M(-2p_i))=h^1( \bigoplus M(-p_i))$. Then  $\text{ker}(H^1(g))=0$. 
\end{proof}
\section{ Non surjectivity}
\label{section2}
In this section we are going to prove theorem \ref{teononsurjfgh}.
We proceed in a similar  way as in \cite{fs1}: we will obtain the non-surjectivity result  applying theorem (\ref{L'vosky}) and a result about very ampleness of line bundles on the blow-up of Enriques surfaces.

Let $S$ be an Enriques surface. First we recall the definition of two positivity measures: the $\phi-$function and the Seshadri constant of a big and nef line bundle $H$ on $S$. The first one is defined as 
	$$
	\phi(H):=\text{min}\{|H \cdot F|: F\in \Pic(S), F^2=0, F \not \equiv 0\},
	$$
where $\equiv$ denotes the numerical equivalence relation. Now set $\epsilon(H,x):=\inf\limits_{x \in C} \frac{H \cdot C}{\text{mult}_{x}C} $, where the infimum is taken over all curves $C$ passing trough $x$. The Seshadri constant $\epsilon(H)$ of $H$ is defined as 
 $$
	\epsilon(H):=\inf\limits_{x \in X} \epsilon(H,x).
	$$
The following holds:
 \begin{equation}
 \label{dis}
0\leq \epsilon(H)^2 \leq \phi(H)^2\leq H^2. 
 \end{equation}
For background and proofs see for example 
\cite{dolgachev}. Now let $\sigma: S' \rightarrow S$ be the blow-up in a point $p$. 
 We will now give, in terms of $\phi$, sufficient conditions for a line bundle of the form $\sigma^*H-lE$ to be big and nef. 

In the following, when we say a ``very general" Enriques surface, we mean that as a point in the moduli space of Enriques surfaces, it lives outside a countable union Zariski-closed subsets. We also recall that a  nodal Enriques surface is one which contains $-2$ curves. In the moduli space of Enriques surface theese correspond to a divisor. An Enriques surface not containing a $-2$ curve is usually called unnodal. 
%
\begin{proposition}
\label{lemminorte}
Let $S$ be an Enriques surface and and $l \geq 1$ be an integer. Let $H$ be a big and nef line  bundle on $S$ 
and suppose one of the following holds:
\begin{itemize}
    \item[i)] $S$ is a very general Enriques surface, $\phi(H)=l$ and $H$ is not of the type $H \equiv \frac{l}{2}(E_1+E_2)$ with $E_i$, $i=1,2$, effective isotropic vectors such that $E_1 \cdot E_2=2$.
    \item[ii)] $S$ is a very general Enriques surface and $\phi(H) \geq l+1$. 
    \item[iii)] $S$ is unnodal and $\phi(H) \geq 2l$.
\end{itemize}. Then $\sigma^*H-lE$ is big and nef.
\end{proposition}
\begin{proof}
First we show that $\sigma^*H-lE$ is nef. From \cite{laz2}, proposition 5.1.5.,  it  follows that $\sigma^*(H)-lE$ is nef  if and only if $\epsilon(H) \geq l$. 
In \cite{galk}, Theorem $1.3$, it is shown that if $S$ is a very general Enriques surface then $\phi(H)=\epsilon(H)$. Then, in case $i)$ or $ii)$ we conclude. Now suppose we are in situation $iii)$. From 
\cite{galk}, Corollary $4.5$, it follows that 
$\epsilon(H) \geq \frac{1}{2}\phi(H) \geq l$ and we immediately conclude. 

From  \ref{dis} and the hypothesis $l \geq 1$ in case $ii$ and $iii)$ we get $H^2\geq \phi(H)^2>0$ and hence $\sigma^*H-lE$ is also big.  Consider now the situation $i)$ and suppose that $\sigma^*H-lE$ is not big, i.e. $H^2=l^2$. Then, again by  \ref{dis}, we have $H^2=\phi^2=l^2$. By \cite{knun}, Proposion $1.4$, we must have $H \equiv l(E_1+E_2)$, where $E_i$, $i=1,2$ are isotropic effective divisor such that $E_1 \cdot E_2=2$. 
\end{proof}
The proof of the next result is a direct application of Reider's Theorem (see \cite{reider}, Theorem $1$, or \cite{dolgachev}, Theorem $2.4.5$).
\begin{proposition}
\label{prop very amp}Let $l \geq 1$ and let $(S,H)$ be a polarized Enriques surface. Suppose that either
\begin{itemize}

\item[(i)] $S$ is a very general Enriques surface, $l \geq 1$ and $\phi(H) \geq \sqrt{2}(l+2)$, or 
   
 \item[(ii)] $S$ in unnodal, $l=1$ and $\phi(H) \geq 3\sqrt{2}$, or $l \geq 2$ and $\phi(H) \geq2(l+1)$.  
\end{itemize}
Let $\sigma: S' \rightarrow S$ be the blow-up in a point and  $E$ be  the exceptional divisor. Then $\sigma^*H-lE$ is a very ample line bundle  on $S'$.
\end{proposition}
\begin{proof}
First observe that $\sigma^*H-lE=\sigma^*(H+K_{S})-(l+1)E + K_{S'}$. Set $H'=H+K_{S}$. By proposition \ref{lemminorte} $\sigma^*H' -(l+1)E$ is big and nef. Indeed $\phi(H')=\phi(H)\geq \sqrt{2}(l+2)\geq l+2$ in case $(i)$ and $\phi(H') \geq 2(l+1)$ in case $(ii)$.  Observe that is also effective. Indeed suppose by contradiction it is not.  Then, by Riemann-Roch and Serre duality, $K_{S'} \otimes (\sigma^*H' -(l+1)E)^{\vee}= -(\sigma^*H-(l+2)E)$ is effective.  Now take a nef effective divisor  $L$ in $S$. Since $\sigma^*L$ is also nef we obtain $ 0\leq \sigma^*L \cdot (-(\sigma^*H-(l+2)E)) =  -L \cdot H <0 $, where the latter is just the fact that $H$ is ample and $L$ effective. Then we conclude that $\sigma^*H' -(l+1)E$ is effective. Now suppose by contradiction that $\sigma^*H -lE$ is not very ample.  Since $\sigma^*H' -(l+1)E$ is an effective, big and nef divisor and $H^2\geq\phi^2(H) \geq 9+(l+1)^2$ in both cases $(i)$ and $(ii)$, we can apply Reider's theorem. 
%
%
%
Then there exists a non trivial effective divisor $D$ in $S'$ 
such that either one of the following holds:
\begin{itemize}
    \item[(a)] $D^2=0$ \ \text{and} \ $(\sigma^*H'-(l+1)E)D \leq 2$;
    \item[(b)] $D^2=-1$ \ \text{and} \ $(\sigma^*H'-(l+1)E)D \leq 1$;
    \item[(c)] $D^2=-2$ \ \text{and} \ $(\sigma^*H'-(l+1)E)D=0$;
    \item[(d)] $(\sigma^*H'-(l+1)E)^2=9$, $D^2=1$ \ \text{and} \ $(\sigma^*H'-(l+1)E)\equiv 3D$ \ \text{in} \ $Num(X)$.
\end{itemize}
Now we show that none of these can happen.

Let $D\sim \sigma^*L-aE$, for some $L \in \text{Pic}(S)$ and  $a \in \mathbb{Z}$. 
Suppose we are in case $(a)$. Then we have $H'L \leq (l+1)a + 2$ and $L^2=a^2$ and so we obtain the following inequalities:
\begin{equation}
\label{inequalitiess}
\phi(H')^2a^2 \leq H'^2a^2=H'^2L^2 \leq (H'\cdot L)^2 \leq ((l+1)a+2)^2,
\end{equation}
where the second inequality follows by Hodge index theorem. If $ |a| \geq 2$ we obtain 
$$
\phi(H') \leq |\frac{(l+1)a+2}{a}| \leq (l+1)+|\frac{2}{a}| \leq (l+1)+1,
$$
which contradicts the hypothesis. If $|a|=1$ from \ref{inequalitiess} we get 
$\phi(H)=\phi(H') \leq (l+1)+2$ which again is not possible. 
If $a=0$ we get $D=\sigma^*L$ with $L$ effective, not numerically trivial and such that $L^2=0$ and $H'L \leq 2$. 
This gives $\phi(H) \leq 2$ and we conclude.

Suppose now we are in case $(b)$. As before one have $L^2=a^2-1$, $H'L \leq a(l+1) +1$. Therefore we obtain
$$
\phi(H')^2(a^2-1) \leq H'^2(a^2-1)=H'^2L^2 \leq (H' \cdot L)^2 \leq  (a(l+1)+1)^2. 
$$
If $|a| \geq 2$ we find $\phi(H') < \sqrt{2}(l+2)$. If $a=1$ then $L$ is an effective divisor such that $L^2=0$ and $H'L \leq l + 2$. Moreover observe that $L$ is not numerically  trivial since otherwise $D \equiv -E$, which is not possibile because $D$ is an effective non trivial divisor. Therefore we obtain $\phi(H') \leq l+ 2$. $a=-1$ cannot happen if $l \geq 1$ because $H'$ is nef and $L$ is effective and $L \cdot H'=-l$. 
If  $a=0$ then $L^2=-1$. This is not possibile for Enriques surfaces.

Suppose now we are in case $(c)$. Then $H'L= a(l+1)$ and $L^2=a^2-2$. Then, as before,
$$
\phi(H')^2(a^2-2) \leq H'^2(a^2-2)=H'^2L^2 \leq (H' \cdot L)^2 \leq  a^2(l+1)^2.
$$
Observe that if  $|a| \geq 2$ this gives  $\phi(H') \leq \sqrt{2}(l+1)$ and hence we conclude. Observe that $|a|=1$ cannot happen because otherwise $L^2=-1$ and this, again, is not possible. 
If $a=0$ then $L$ is a effective divisor such that $L^2=-2$ and $H' L=0$. This cannot happen because $H' \cdot L=(H+K_S) \cdot L$, $H$ is ample and $L$ is effective.

Suppose we are now in case $(d)$. Then $H'^2=9+(l+1)^2$ which is not possible since $H'^2 \geq \phi(H)^2>9+(l+1)^2$ by hypothesis. 
\end{proof}

\begin{cor}\label{corfond}
With the same hypothesis of the previous result we have $\sigma^*H-lE + \sigma^*K_S=\sigma^*(H+K_S)-lE $ is very ample on $X$.
\end{cor}
\begin{proof}
Apply proposition \ref{prop very amp} with $H+K_S $ instead of  $H$.
\end{proof}
We observe that proposition \ref{prop very amp} has the following corollary.
\begin{cor}
   \label{cor very amp}Let $l \geq 2$ and let $(S,H)$ be a polarized Enriques surface. Suppose that either
\begin{itemize}

\item[(i)] $S$ is a very general Enriques surface 
and $\phi(H) \geq \sqrt{2}(l+2)$, or 
   
 \item[(ii)] $S$ is unnodal 
and $\phi(H) \geq2(l+1)$.  
\end{itemize}
Then there exists a curve $C$ in the linear system $|H|$ with an $l-$ordinary point. 
\end{cor}
Now we conclude with the proof of theorem \ref{teononsurjfgh}.
\begin{proof}[Proof of theorem \ref{teononsurjfgh}]
Let $\sigma: S' \rightarrow S$ be the blow-up in  a point and  $E$ the exceptional divisor. By the universal property of normalization we can suppose $ C \in |\sigma^*H-dE|$ and $\alpha=\sigma^*K_{S_{|_{C}}}$. From proposition \ref{prop very amp}
it follows that $\mathcal{O}_C(C)=\omega_C-T_d+\alpha $ 
is very ample. Observe that $H^0(C,\mathcal{O}_{C}(C))=H^0(S',\mathcal{O}_{S'}(C))-1$. 
%
Applying theorem \ref{teononsurjfgh} we obtain that  $\Phi_{\omega_C,\omega_C - T_d + \alpha}$  is not surjective. 

Now we want to prove that also $\Phi_{\omega_C -T_d,\omega_C - T_d + \alpha}$ is not surjective using proposition \ref{propositionfundamentalsurjection} with $L=\omega_C$, $M=\omega_C - T_d + \alpha$ and $n=d$. Observe that since $\mathcal{O}_{S'}(C)+K_{S'}=(\sigma^*(H+K_S)-(d-1)E)$, $\omega_C=\mathcal{O}_{S'}(C+\omega_{S'})$ is very ample by by corollary \ref{corfond}. Analogously $\mathcal{O}_C(C+\sigma^*K_S)=\omega_C-T_d$ is very ample. It remains to show that $h^1(\omega_C)=h^1(\omega_C-T_d)$ or equivalently that $h^0(\omega_C-T_d)=h^0(\omega_C)-d$. Consider then the following commutative diagram:
\begin{center}
\begin{tikzcd}
& 0 \arrow[d] & 0 \arrow[d] & 0 \arrow[d]\\
    
0 \arrow[r] &  \mathcal{O}_{S'}(K_S'-E) \arrow[r,] \arrow[d,]
    & \mathcal{O}_{S'}(K_{S'}) \arrow[r,] \arrow[d,] & \mathcal{O}_{E}(K_{S'}) \arrow[r]\arrow[d] & 0 \\
    
0  \arrow[r]    &  \mathcal{O}_{S'}(C+K_{S'}-E)\arrow[d] \arrow[r,]
&  \mathcal{O}_{S'}(C+K_{S'})\arrow[d] \arrow[r] & \mathcal{O}_{E}(C+K_{S'})\arrow[d] \arrow[r] & 0 \\

0  \arrow[r]    &  \mathcal{O}_{C}(\omega_C-T_d)  \arrow[d] \arrow[r,]
&  \mathcal{O}_{C}(\omega_C) \arrow[d] \arrow[r] & \bigoplus_{i=1}^d\mathcal{O}_{p_{i}} \arrow[d] \arrow[r] & 0\\
& 0 & 0 & 0 & 
\end{tikzcd}
\end{center}
and the one induced at level of global section:
\begin{center}
\begin{tikzcd}

&  0  \arrow[d,]
    & 0  \arrow[d,] & 0 \arrow[d] & &  \\
    
0  \arrow[r]    &  H^0(\mathcal{O}_{S'}(C+K_{S'}-E))\arrow[d] \arrow[r,]
&  H^0(\mathcal{O}_{S'}(C+K_{S'}))\arrow[d] \arrow[r] & \mathbb{C}^d\arrow[d] \arrow[r] & 0  &  \\

0  \arrow[r]    &  H^0(\mathcal{O}_{C}(\omega_C-T_d))   \arrow[r,]
&  H^0(\mathcal{O}_{C}(\omega_C) \arrow[d] \arrow[r]) & \mathbb{C}^d \arrow[d]  &   & \\
&  & 0 & 0& &
\end{tikzcd}
\end{center}
where we are using that $H^0(\mathcal{O}_{S'}(K_{S'})) \simeq H^0(\mathcal{O}_{S}(K_{S}))=0$,  $E \simeq \mathbb{P}^1$ and $\mathcal{O}_{E}(K_{S'}) $ is a divisor of degree $-1$ in $E \simeq \mathbb{P}^1$, $h^1(\mathcal{O}_{S'}(C+K_{S'}-E))=0$ by Kawamata vanishing theorem since $\mathcal{O}_{S'}(C-E)=\mathcal{O}_{S'}(\sigma^*H-(d+1)E)$ is big and nef. $H^1(K_{S'})\simeq H^1(\mathcal{O}_{S'})=0$ because $S$ is an Enriques surface and $S'$ is a blow-up. Hence  we conclude that $h^0(\mathcal{O}_{C}(\omega_C-T_d))=h^0(\mathcal{O}_{C}(\omega_C)) -d.$ 
\end{proof}

\section{Surjectivity for special curves}
\label{section3}
We start this section with a proposition giving sufficient conditions for the surjectivity of mixed Gaussian maps for surfaces which are the product of two curves. The idea of computing the rank of gaussian maps on the product of two curves with gaussian maps on the two factors dates back to Wahl (\cite{wahl2}, Lemma $4.12$). See also  Colombo-Frediani (\cite{colfre}, Theorem $3.1$) for the second Wahl map. 

\begin{proposition} \label{propsurjmappasopra}
    Let $X=C_1 \times C_2$. Let  $D_i$, $i=1,2$  be  effective divisors on $C_i$. Let $p_i:X=C_1 \times C_2 \rightarrow C_i$, $i=1,2$ be the projections. Let $L_i$ and $M_i$ be line bundles on $C_i$, $i=1,2$, such that $\text{deg}(L_i),  \text{deg}(M_i) \geq  2g_i+2$  and $\text{deg}(L_i) +  \text{deg}(M_i) \geq 6g_i+3$, 
     for $i=1,2$. Let $L=p_1^*L_1 \otimes p_2^*L_2$ and $M=p_1^*M_1\otimes p_2^*M_2$. Then    $\Phi_{X,L,M}$ is surjective. 
\end{proposition}

\begin{proof}
We want to relate the gaussian map $\Phi_{X,L,M}$ with gaussian maps on $C_i$, $i=1,2$. Let $q_i: X \times X \rightarrow X$, $i=1,2$ the two projections. Recall that $\Phi_{X,L,M}$ is given by:
$$
\Phi_{X,L,M}: H^0(X \times X, I_{\Delta_{X \times X}} \otimes q_1^*L \otimes  q_2^*M) \rightarrow H^0(X \times X, I_{\Delta_{X \times X}} /I_{\Delta_{X \times X}}^2   \otimes q_1^*L \otimes  q_2^*M) 
$$
Let $q_{i,1}:C_{1} \times C_{1} \ra C_1 \ \text{for}  \ i=1,2 
$
the projections and analougly
$
q_{i,2}:C_{2} \times C_{2} \ra C_2$. Let $(\phi_1,\phi_2)$ the isomorphism which exchange factors:
$$
X \times X= (C_1 \times C_2) \times  (C_1 \times C_2 ) \xrightarrow{(\phi_1,\phi_2)} (C_1 \times C_1) \times  (C_2 \times C_2 ),
$$
i.e. $\phi_i((x_1,x_2),(y_1,y_2)=(x_i,y_i)$.
Observe that  
$$
I_{\dg{X \times X}}= \phi_1^{*}I_{\dg{C_1 \times C_1}} + \phi_2^{*}I_{\dg{C_2 \times C_2}},
$$
where $\phi_i^{*}I_{\dg{C_i \times C_i}}$, $i=1,2$, are the inverse image ideal sheaves or equivalently the pullbacks sheaf (because projections are flat). Now consider the isomorphism of $\mathcal{O}_{X}$-modules:
$$
\Omega^1_{X} \simeq I_{\dg{X \times X}} \otimes_{X\times X} \mathcal{O}_{\dg{X \times X}}.
$$
Under this identification the decomposition
$$
\Omega^1_{X}\simeq p_1^*\Omega^1{_{C_1}} \oplus p_2^*\Omega^1{_{C_2}}
$$
can be read as 
$$
I_{\dg{X \times X}} \otimes_{X \times X} \mathcal{O}_{\dg{X \times X}}\simeq(\phi_1^{*}I_{\dg{C_1 \times C_1}} \oplus \phi_2^{*}I_{\dg{C_2 \times C_2}}) \otimes_{X \times X} \mathcal{O}_{\dg{X \times X}}.
$$
So we obtain  the following commutative diagram:
\begin{center}
\begin{tikzcd}
( \phi_1^{*}I_{\dg{C_1 \times C_1}} \oplus \phi_2^{*}I_{\dg{C_2 \times C_2}}) \otimes q_1^*L \otimes q_2^*M    \arrow[r, ""] \arrow[d, ""]
    & ( \phi_1^{*}I_{\dg{C_1 \times C_1}} \oplus \phi_2^{*}I_{\dg{C_2 \times C_2}}) \otimes q_1^*L \otimes q_2^*M \otimes \mathcal{O}_{\dg{X \times X} }\
    \arrow[d, "\simeq"] 
    \\
 I_{\dg{X \times X}} \otimes q_1^*L \otimes q_2^*M  \arrow[r, black, "" black] 
&I_{\Delta_{X \times X}} /I_{\Delta_{X \times X}}^2 \otimes q_1^*L \otimes q_2^*M   \end{tikzcd}
\end{center}
Taking global section we obtain
\begin{center}
\begin{tikzcd}
H^0(\phi_1^{*}I_{\dg{C_1 \times C_1}} \oplus \phi_2^{*}I_{\dg{C_2 \times C_2}} \otimes q_1^*L \otimes q_2^*M )   \arrow[r, "\psi"] \arrow[d, ""]
    & H^0(( \phi_1^{*}I_{\dg{C_1 \times C_1}} \oplus \phi_2^{*}I_{\dg{C_2 \times C_2}}) \otimes q_1^*L \otimes q_2^*M) \otimes \mathcal{O}_{\dg{X \times X}})\
    \arrow[d, " \ \simeq"] 
    \\
 H^0(I_{\dg{X \times X}} \otimes q_1^*L \otimes q_2^*M ) \arrow[r, black, "\Phi_{X,L,M}" black] 
&H^0(I_{\Delta_{X \times X}} /I_{\Delta_{X \times X}}^2 \otimes q_1^*L \otimes q_2^*M)   \end{tikzcd}
\end{center}

%
In order to show that $\Phi_{X,L,M}$ is surjective we will show the surjectivity of  $\psi$. Cleary $\psi$ is surjective if each of the direct sum map is surjective:
$$
\psi_1:H^0(\phi_1^{*}I_{\dg{C_1 \times C_1}} \otimes q_1^*L \otimes q_2^*M ) \rightarrow 
 H^0(( \phi_1^{*}I_{\dg{C_1 \times C_1}}  \otimes q_1^*L \otimes q_2^*M) \otimes \mathcal{O}_{\dg{X \times X}})
 $$
 and 
$$
\psi_2:H^0(\phi_2^{*}I_{\dg{C_2 \times C_2}} \otimes q_1^*L \otimes q_2^*M ) \rightarrow 
 H^0((  \phi_2^{*}I_{\dg{C_2 \times C_2}}) \otimes q_1^*L \otimes q_2^*M) \otimes \mathcal{O}_{\dg{X \times X}})
 $$
 Let us deal with the first map. The same reasoning will apply also to the second one. 
Observe that
$$
p_j \circ q_i=q_{i,j} \circ \phi_j.
$$
Then we can write
\begin{align}
\label{identifi}
    q_1^*L \otimes  q_2^*M&=q_1^*(p_1^*L_1 \otimes p_2^*L_2 ) \otimes q_2^*(p_1^*M_1 \otimes p_2^*M_2 )
    \\
    &=\phi_1^*(q_{1,1}^*L_1 \otimes q_{2,1}^*M_1)\otimes \phi_2^*(q_{1,2}^*L_2 \otimes q_{2,2}^*M_2).
\end{align}
And so we obtain 
$$
\phi_1^{*}I_{\dg{C_1 \times C_1}} \otimes q_1^*L \otimes q_2^*M 
    \simeq \phi_1^*(I_{\dg{C_1 \times C_1}} \otimes (q_{1,1}^*L_1 \otimes q_{2,1}^*M_1) ) \otimes \phi_2^*((q_{1,2}^*L_2 \otimes q_{2,2}^*M_2))
        $$
Using that $\mathcal{O}_{\dg{X \times X}} \simeq 
\phi_1^*\mathcal{O}_{\dg{C_1  \times C_1 }} \otimes \phi_2^*\mathcal{O}_{\dg{C_2  \times C_2 }} $ we also obtain 
\begin{align*}
\phi_1^{*}I_{\dg{C_1 \times C_1}}  \otimes q_1^*L \otimes q_2^*M  \otimes \mathcal{O}_{\dg{X \times X}} 
\end{align*}
\begin{align*}
        \simeq \phi_1^*(I_{\dg{C_1 \times C_1}} \otimes (q_{1,1}^*L_1 \otimes q_{2,1}^*M_1) \otimes \mathcal{O}_{\dg{C_1  \times C_1 }}  ) \otimes \phi_2^*((q_{1,2}^*L_2 \otimes q_{2,2}^*M_2) \otimes \mathcal{O}_{\dg{C_2  \times C_2 }})
  \end{align*}
So $\psi_1$ becomes a map:
\begin{center}
    
\begin{tikzcd}
    H^0( \phi_1^*(I_{\dg{C_1 \times C_1}} \otimes (q_{1,1}^*L_1 \otimes q_{2,1}^*M_1) ) \otimes \phi_2^*((q_{1,2}^*L_2 \otimes q_{2,2}^*M_2))) 
    \arrow[d]
    \\
        H^0(\phi_1^*(I_{\dg{C_1 \times C_1}} \otimes (q_{1,1}^*L_1 \otimes q_{2,1}^*M_1) \otimes \mathcal{O}_{\dg{C_1  \times C_1 }}  ) \otimes \phi_2^*((q_{1,2}^*L_2 \otimes q_{2,2}^*M_2) \otimes \mathcal{O}_{\dg{C_2  \times C_2 }}))
    \end{tikzcd}
\end{center}
Now using that $X \times X \xrightarrow[]{\simeq}(C_1 \times  C_1) \times (C_2 \times  C_2)$ and Künneth formula we get:
\begin{align*}
     &H^0(X \times X, \phi_1^*(I_{\dg{C_1 \times C_1}} \otimes (q_{1,1}^*L_1 \otimes q_{2,1}^*M_1) ) \otimes \phi_2^*((q_{1,2}^*L_2 \otimes q_{2,2}^*M_2))) 
     \\
     \simeq& H^0(C_1 \times C_1, I_{\dg{C_1 \times C_1}} \otimes (q_{1,1}^*L_1 \otimes q_{2,1}^*M_1) ) \otimes H^0(C_2 \times C_2, (q_{1,2}^*L_2 \otimes q_{2,2}^*M_2)),
\end{align*}
and
\begin{align*}
     &H^0(X \times X, \phi_1^*(I_{\dg{C_1 \times C_1}} \otimes (q_{1,1}^*L_1 \otimes q_{2,1}^*M_1) \otimes \mathcal{O}_{\dg{C_1\times C_1}}) \otimes \phi_2^*((q_{1,2}^*L_2 \otimes q_{2,2}^*M_2))\otimes \mathcal{O}_{\dg{C_2\times C_2}}) 
     \\
     \simeq& H^0(C_1 \times C_1, I_{\dg{C_1 \times C_1}} \otimes (q_{1,1}^*L_1 \otimes q_{2,1}^*M_1) \otimes \mathcal{O}_{\dg{C_1\times C_1}}) ) \otimes H^0(C_2 \times C_2, (q_{1,2}^*L_2 \otimes q_{2,2}^*M_2) \otimes \mathcal{O}_{\dg{C_2\times C_2}})).
\end{align*}
Under these identifications   $\psi_1$ becomes:
\begin{center}
\begin{tikzcd}
    H^0(I_{\dg{C_1 \times C_1}} \otimes q_{1,1}^*L_1 \otimes q_{2,1}^*M_1) \otimes H^0( q_{1,2}^*L_2 \otimes q_{2,2}^*M_2) 
    \arrow[d,"\psi_1"]
    \\
        H^0(I_{\dg{C_1 \times C_1}} \otimes q_{1,1}^*L_1 \otimes q_{2,1}^*M_1 \otimes \mathcal{O}_{\dg{C_1  \times C_1 }}  )\otimes H^0(q_{1,2}^*L_2 \otimes q_{2,2}^*M_2 \otimes \mathcal{O}_{\dg{C_2  \times C_2 }})
    \end{tikzcd}
\end{center}
and it is given by the tensor product $ \Phi_{C_1,L_1, M_1} \otimes \Phi^0_{C_2,L_2,M_2}$,  where  
$$
\Phi_{C_1,L_1, M_1}:H^0(I_{\dg{C_1 \times C_1}} \otimes q_{1,1}^*L_1 \otimes q_{2,1}^*M_1) \rightarrow  H^0(I_{\dg{C_1 \times C_1}} \otimes q_{1,1}^*L_1 \otimes q_{2,1}^*M_1 \otimes \mathcal{O}_{\dg{C_1  \times C_1 }}  )
$$
and
$$
\Phi^0_{C_2,L_2,M_2}:H^0( q_{1,2}^*L_2 \otimes q_{2,2}^*M_2)  \rightarrow H^0(q_{1,2}^*L_2 \otimes q_{2,2}^*M_2 \otimes \mathcal{O}_{\dg{C_2  \times C_2 }}).$$
Analogously one can show that $\psi_2=\Phi^0_{C_1,L_1, M_1} \otimes \Phi_{C_2L_2,M_2}.$ Therefore we obtain
\begin{equation}
  \psi=  
  \Phi_{C_1,L_1, M_1} \otimes \Phi^0_{C_2,L_2,M_2} \oplus \Phi^0_{C_1,L_1, M_1} \otimes \Phi_{C_2L_2,M_2}.
\end{equation}
 Now observe that if deg$(L_i)$, deg$(M_i) \geq 3g_i + 2$ for $i=1,2$, then by Theorem 1.1 of \cite{einlaz}, each gaussian map is surjective and so $\psi$ is.
\end{proof}
\begin{remark}
Let $X_1$ and $X_2$ be two smooth varieties of any dimension. Let $L_1$, $M_1$ and $L_2$, $M_2$ be two line bundles on $X_1$ and $X_2$ respectively. Denote by $L=L_1 \boxtimes L_2$ and $M=M_1 \boxtimes M_2$. We observe that a similar proof gives a lifting of $\Phi_{L,M}$ by $\Phi_{X_1,L_1,M_1} \otimes \Phi^0_{X_2,M_2,L_2} \oplus \Phi^0_{X_1,L_1,M_1} \otimes \Phi_{X_2,M_2,L_2}$. 
\end{remark}
We are now going to prove a surjectivity result for mixed gaussian maps on curves living in the product of two curves.

\begin{proposition}\label{princsurjaltra}
With the same hypothesis and notations of proposition \ref{propsurjmappasopra}, let  $C$ be  an effective smooth curve in the linear system $|p_1^*D_1 + p_2^*D_2|$. Denote by $l_i$ and $m_i$ the degree of $L_i$ and $M_i$ respectively.  Moreover suppose that 
\begin{enumerate}
\item $l_i, m_i \geq  2g_i+2$  and $l_i+ m_i \geq 6g_i+3$;
\item $l_i+m_i> 2g_i-2+d_i$ for $i=1,2$\label{punto1};

\item $d_2(l_1+m_1-(2g_1-2))+d_1(l_2+m_2-(2g_2-2)) -4d_1d_2 >0$. \label{punto2} 
\end{enumerate}
Then 
    $$
    \Phi_{C,L_{|_{C}} \otimes M_{|_{C}}}
    $$
    is surjective.

\end{proposition}
 \begin{proof}
   Consider the following  commutative diagram \begin{equation}\label{diagramma}
		\begin{tikzcd}
			H^0(X\times X,\mathcal{I}_{\Delta_X}\otimes L\boxtimes M ) \arrow{dd}\arrow{r}{\Phi_{L, M }}&  H^0(X,   \Omega^1_X \otimes L \otimes M )\arrow{dr}{\pi_1}&\\
			&  & H^0(C,   \Omega^1_X \otimes L \otimes M\restr{C})\arrow{dl}{\pi_2}\\
			H^0(C\times C,\mathcal{I}_{\Delta_C} \otimes  L_{|_{C}} \boxtimes M_{|_{C}})\arrow{r}{\Phi_{L_{|_{C}},M_{|_{C}}} }&H^0(C,\omega_C \otimes L_{|_{C}} \otimes M_{|_{C}}).&
		\end{tikzcd}
	\end{equation}
 Observe that the vertical arrow and $\pi_1$ are restriction maps, whereas  $\pi_2$ comes from the conormal bundle sequence
	\begin{equation*}\label{symmetric conormal}
		0\ra \mathcal{O}_C(-C)\ra \Omega^1_{X_{\restr{C}}} \ra\omega_C\ra 0
	\end{equation*}
	tensored by $\mathcal{O}_C(L+M)$.
We prove that $\Phi_{L,M}, \pi_1, \text{and } \pi_2$ are surjective. From this we obtain the desired surjectivity result. 
The surjectivity of $\Phi_{L,M}$  is just \ref{propsurjmappasopra}. 
The surjectivity of $\pi_1$  will follow from the vanishing  of $H^1(X,\Omega_X \otimes L \otimes M(- C)) \simeq 
H^1(X,p_1^*\omega_{C_1} \otimes L \otimes M(- C)) \oplus H^1(X,p_1^*\omega_{C_2} \otimes L \otimes M(- C)).$ Consider the first piece. Observe that 
$$
H^1(X,p_1^*\omega_{C_1} \otimes L \otimes M(- C)) \simeq H^1(X,p_1^*(\omega_{C_1} \otimes L_1 \otimes M_1(-D_1)) \otimes p_2^*(L_2 \otimes M_2(- D_2)).
$$
By Künneth this is just
\begin{align*}
   H^0(C_1,\omega_{C_1} \otimes L_1 \otimes M_1(-D_1)) & \otimes H^1(C_2,L_2 \otimes M_2(- D_2)).
    \\
    & \oplus
    \\
  H^1(C_1,\omega_{C_1} \otimes L_1 \otimes M_1(-D_1)) & \otimes H^0(C_2,L_2 \otimes M_2(- D_2)).
 \end{align*}
Now observe that $h^1(C_2,L_2 \otimes M_2(- D_2))=0$ and $h^1(C_1,\omega_{C_1} \otimes L_1 \otimes M_1(-D_1))$ are zero by Serre duality and the hypothesis \ref{punto1}. Analogously $H^1(X,p_1^*\omega_{C_2} \otimes L \otimes M(- C))$  decomposes as 
\begin{align*}
   H^0(C_1, L_1 \otimes M_1(-D_1)) & \otimes H^1(C_2, \omega_{C_2}  \otimes L_2 \otimes M_2(- D_2)).
    \\
    & \oplus
    \\
  H^1(C_1,  L_1 \otimes M_1(-D_1)) & \otimes H^0(C_2,  \omega_{C_2} \otimes L_2 \otimes M_2(- D_2)).
\end{align*}
Again, $h^1(C_2, \omega_{C_2}  \otimes L_2 \otimes M_2(- D_2))$ and $h^1(C_1, L_1 \otimes M_1(-D_1))$ are zero by Serre duality and the hypothesis the hypothesis \ref{punto1}.  The surjectivity of $\pi_2$ will follow from the vanishing of $H^1(C,(L_{|_{C}}+L_{M|_{C}} -C_{|_{C}}).$ By Serre duality it will be enough to show that 
 $$
 \text{deg}(K_C+C_{|_{C}}-L_{|_{C}}-M_{|_{C}}) <0.
 $$
 This is just hypothesis \ref{punto2}. 
 \end{proof}

\begin{mainconstruction}
\label{curvsrufapic}
In this remark we consider a construction we will use in the following corollary. First observe that if $S$ is a smooth surface, $H$ is an ample divisor on $S$ and $C \in |H|$ is a smooth curve, then the restriction map
$$
\text{Pic}^0_S \rightarrow \text{Pic}^0_C 
$$
is injective by Lefshetz hyperlane theorem (see for example \cite{fujita}, theorem C).

Now let $C_1$ and $C_2$ be two curves. Let $X$ be the product $C_1 \times C_2$. 
Let $p_i:X \rightarrow C_i$, $i=1,2$ be the two projections and let $D_i$ be effective divisors of degree $d_i$ such that $|p_1^*D_1 + p_2^*D_2|$. is base point free. Let $C$ be a  smooth irreducible curve in the linear system $|p_1^*D_1 + p_2^*D_2|$.   Let $\alpha' \in \Pic^0(C_1)$ a non trivial $2$-torsion element( in particular $g(C_1) \geq 1$ ).  Then $\alpha_1:=p_1^*\alpha'$ is a non trivial  $2$-torsion element in Pic$(X)$ and $\alpha:=\alpha_{1_{|_{C}}}$ is a non trivial $2$-torsion element in $\Pic(C)$. 
Let supp$(D_1)=\{p_{1,1},...,p_{1,d_1}\}$ and denote by $T_{d_2}$ the divisor of the $d_2$  intersection points between the fiber $p_{1}^{-1}(p_{1,1})$ and $C$. 

\end{mainconstruction}

\begin{remark}
\label{remajdjdjs}
 We observe that a sufficient condition for $\mathcal{O}_X(p_1^*D_1 + p_2^*D_2)$ to be base-point free is that both $\mathcal{O}_{C_1}(D_1)$ and $\mathcal{O}_{C_2}(D_2)$ are. Observe that if $C$ is any curve of genus $g \geq 1$, a  general effective divisor $D$ of degree $d \geq g+1$ is basepoint-free. This follows from classical results but we recall it. 
 
Since every divisor of degree $2g$ is  base-point-free, we can restrict to the case $g\geq 2$  and $ g+1 \leq d \leq 2g-1$. Consider first the case $d=2g-1$. Let $D'$ be a general divisor of degree $2g-2$ and $p \in C$ be a point. Then, by Riemann-Roch, it immediately follows that $D'+p$ is a base-point free divisor of degree $2g-1$.  Now suppose $ g+1 \leq d \leq 2g-2$ and consider the  Brill Noether variety $W^{r}_{d}$  parametrizing (isomorphism classes of) line bundles of degree $d$ with  the dimension of space of global sections greater or equal then $r-1$. Since $d$ is greater then $g+1$, by Riemann-Roch,  $Pic^{d}(C)=W^{d-g}_{d}$. Hence we have to show that a general element of $W^{d-g}_{d}$, with $ g+1 \leq d \leq 2g-2$, is base-point-free. Line bundles with base points are given, inside $W^{d-g}_{d}$, by the image of the natural map
\begin{equation}
\label{equass1}
 W^{d-g}_{d-1} \times  W^{0}_{1} \rightarrow W^{d-g}_{d}. 
\end{equation}
Consider the isomorphism $W^{d-g}_{d-1} \simeq W^0_{2g-1-d}$ given by $L \rightarrow \omega_C \otimes L^{\vee}$. 
 Since $0 \leq 2g-1-d \leq  g$, the last one is birational to $\operatorname{Sym}^{2g-1-d}C$ and hence has dimension $2g-1-d$. Then the image of \ref{equass1} has dimension $2g-d$. On the other hand $W^{d-g}_{d}$ has dimension greater than or equal to $\rho(g,d-g,d)=g$. We conclude that if $d \geq g+1$ the image of \ref{equass1} is proper subvariety and hence that the general element is base-point-free.
\end{remark}

\begin{cor}\label{corsurjfirstgauss}
Using the construction  \ref{curvsrufapic}  
suppose that one of the following holds:
    \begin{enumerate}
    \item  $g(C_i) \geq 2$ \ i=1,2, $d_1 \geq 5$, $d_2 \geq 4$, $d_1 \geq g_1+5$, $d_2 \geq g_2+4 $;

    \item $g(C_1)=1$, $g(C_2) \geq 2$, $d_1 \geq 6$, $d_2 \geq 4$,  $d_2 \geq g_2+4$, $d_1 > \frac{d_2}{g_2-1}$;

    \item $g(C_1) \geq 3$, $g(C_2)=1$, $d_1 \geq 5$, $d_2 \geq 5$, $d_1 \geq g_1+5$; 

    \item $C_1=\mathbb{P}^1$, $g(C_2) \geq 2$, $d_1\geq 5$, $d_2 \geq 4$, $d_2 \geq g_2+4$, $d_1(g_2-1)>2d_2$;  

     \item $g(C_1) \geq 3$, $C_2= \mathbb{P}^1$, $d_1\geq 5$, $d_2 \geq 4$, $d_2(g_1-2)>d_1 \geq g_1+5$.  \label{ouno4}
    \end{enumerate}    
    Then 
    $$
    \Phi_{C,\omega_C - T_{d_2},\omega_C-T_{d_2} + \alpha, }
    $$
    and
    $$
    \Phi_{\omega_C,\omega_C-T_{d_2} + \alpha}
    $$
    are surjective. 
\end{cor}
\begin{proof}
Set $L_1=\omega_{C_1}+D_1-p_{1,1}$, $L_2=\omega_{C_2}+D_2$, $M_1=\omega_{C_1}+D_1-p_{1,1}+\alpha'$,  $M_2=\omega_{C_2}+D_2$ and $L'_1=\omega_{C_1}+D_1$,  $L'_2=\omega_{C_2}+D_2$, $M'_1=\omega_{C_1}+D_1-p_{1,1} + \alpha'$ and $M'_2=\omega_{C_2}+D_2$. Denote by $l_i,m_i, \ i=1,2$ and  $l_i',m_i', \ i=1,2$ their degrees. To prove the surjectivity of the gaussian maps  we want to apply proposition \ref{princsurjaltra} with $L_i, M_i, \ i=1,2$ in the first case,  and $L_i', M_i', \ i=1,2$, in the second. Since $l_i' \geq l_i, \ i=1,2$, $m_i' \geq m_i, i=1,2$, it is enough to verify the hypothesis of proposition \ref{princsurjaltra} in the first situation. It is easy to see that the conditions become: $d_1 \geq 5$, $d_2 \geq 4$, $d_1 \geq g_1+5$, $d_2 \geq g_2+4 $ and $d_2(g_1-2)+d_1(g_2-1)>0$. Then we conclude as in the statement. 
\end{proof}
We end this section with a surjectivity result for the related multiplication maps.
\begin{proposition}
\label{surjmult} Using the construction \ref{curvsrufapic} suppose that $d_2 \geq 3$ and $d_1 \geq 4$,   $g_1 \geq 1$, or $g_1=1$ and $d_2 \geq 3$. Then 
\begin{equation}
\label{equazione123}
\Phi^0_{\omega_{C}-T_{d_2},\omega_{C}-T_{d_2}+\alpha} 
\end{equation}
and 
\begin{equation}
\label{equazione1234}
\Phi^0_{\omega_{C},\omega_{C}-T_{d_2}+\alpha} 
\end{equation}
are surjective. 
\end{proposition}
\begin{proof}
Consider first $\Phi^0_{\omega_{C}-T_{d_2},\omega_{C}-T_{d_2}+\alpha}$ 
and denote it by 
$\Phi^0$. Set  $L=K_X+C-p_1^*(p_{1,1})$ and $M=K_X+C-p_1^*(p_{1,1}) + \alpha_1$. Then $L_{|_{C}}=\omega_{C}-T_{d_2}$, $M_{|{C}}=\omega_{C}-T_{d_2}+\alpha$. Consider the following  commutative diagram \begin{equation}
		\begin{tikzcd}
			H^0(X, L) \otimes H^0(X, M)  \arrow{d}\arrow{r}{\Phi^0_{L, M }}&  H^0(X  ,   L \otimes M )\arrow{d}{p}\\
			H^0(C, L_{|_{C}}) \otimes H^0(C, M_{|_{C}})  \arrow{r}{\Phi^0 }&H^0(C, L_{|_{C}} \otimes M_{|_{C}}).
		\end{tikzcd}
	\end{equation}
where $p$ is the restriction map. In order to prove the surjectivity result, again, it is sufficient to prove that $\Phi^0_{L,M}$ and $p$ are surjective. The multiplication map: 
$$
 H^0(X \times X, q_1^*L \otimes q_2^*M)   \xrightarrow{\Phi^0_{L,M}}
    H^0(X \times X, q_1^*L \otimes q_2^*M \otimes \mathcal{O}_{\dg{X \times X} })
    $$
    decomposes, using  the identifications in \ref{identifi} with $L_1=\omega_{C_1} + D_1 -p_{1,1} $ and $L_2=\omega_{C_2} + D_2$, $M_1=\omega_{C_1} + D_1 - p_{1,1} + \alpha_1$, $M_2=\omega_{C_2} + D_2$, and Künneth theorem as before, as the tensor product of  the multiplication maps on the curves $C_i: \ i=1,2$:
    $$
\Phi^0=\Phi^0_{L_1,M_1} \otimes\Phi^0_{L_2,M_2}.
    $$
    Since $l_i,m_i \geq 2g_i+1, \ i=1,2$, each of the multiplication maps is  surjective by a classical result of Mumford. The surjectivity of $p$  will follow from the vanishing  of $H^1(X,L \otimes M - C)$. By Künneth this is isomorphic to
\begin{align*}
    H^0(C_1,L_1   \otimes M_1 -D_1) &\otimes  H^1(C_2,L_2   \otimes M_2 -D_2)
    \\
    & \oplus
    \\
H^1(C_1,L_1   \otimes M_1 -D_1) &\otimes  H^0(C_2,L_2 \otimes M_2 - D_2).
 \end{align*}
Now observe that $h^1(C_2,L_2   \otimes M_2(-D_2))=h^1(C_1,L_1   \otimes M_1(-D_1))=0$.  This is a consequence of Serre duality and the fact that $l_i+m_i> 2g_i-2+d_i$.  
This ends the proof of the surjectivity of \ref{equazione123}. An identical proof, with $L_1=\omega_{C_1}+D_1$,  $L_2=\omega_{C_2}+D_2$, $M_1=\omega_{C_1}+D_1-p_{1,1} + \alpha'$ and $M_2=\omega_{C_2}+D_2$,  gives the surjectivity of \ref{equazione1234}. 
\end{proof}

\section{Some useful lemmas}
\label{section4}
In this section we prove some results we will need in the next one. 
%
%
Let $C$ be a curve. We will need an upper bound on the gonality of curves in the  surface $C \times \mathbb{P}^1$, where $C$ is a curve. The proof is very much inspired by  \cite{hr} (see Lemma 2.8 and Theorem 6.1).

Let $p_1:C \times \mathbb{P}^1 \rightarrow C$, $p_2:C \times \mathbb{P}^1 \rightarrow \mathbb{P}^1$ be the two projections. Let $C_0$ be the class of a fiber of $p_2$.  Recall that 
\begin{align*}
    Pic(C \times \mathbb{P}^1)= p_1^*(Pic(C)) \oplus \mathbb{Z}C_0,
\end{align*}
and that the Néron-Severi is generated by $C_0$ (class of a fiber of $p_2$) and the class of a fiber of $p_1$,  which we will call $f$. We are going to prove the following:
\begin{proposition}
\label{goncurvesprod}
    Let $X \in |p_1^*(D_1)+ d_2C_0|$ be a curve in $C \times \mathbb{P}^1$. Then 
    \begin{itemize}
    
    \item if $C$ is hyperelliptic, 
     \begin{equation*}
    \text{gon}(X) \geq \text{min}(d_1,2d_2).
    \end{equation*}
\item If $C$ is any curve, $g(X) >0$ and $d_2 \geq \frac{d_1}{4}+1+\frac{1}{d_1}$  

\begin{equation*}
    \text{gon}(X) \geq \text{min}(d_1,d_2\text{gon}(C_1)).
    \end{equation*}
    \end{itemize}
\end{proposition}
For the proof we will use the following theorem of Serrano (see \cite{serrano}):
\begin{teo}
Let X be a smooth curve on a smooth surface S. Let $\phi: X \rightarrow \mathbb{P}^1$ be a surjective morphism of degree d. Suppose that either
\begin{itemize}
\item[(a)] $X^2 >(d+1)^2$, or 
\item[(b)]$X^2 >\frac{1}{2}(d + 2)^2$ and $K_S$ is numerically even.
\end{itemize}
Then there exists a morphism $\psi:S \rightarrow \mathbb{P}^1$ such that $\psi_{|_{X}}=\phi$.
\end{teo}
Recall that a divisor $D$ is called numerically even if $D\cdot E$ is even for any other divisor $E$. In our case, being $K_{C \times \mathbb{P}^1}\equiv -2C_0 +(2g(C)-2)f$, we have that $K_{C \times \mathbb{P}^1}$ is numerically even. Before giving the proof of proposition \ref{goncurvesprod}  we will need the folllowing:
\begin{lemma}
\label{lemmino}
Let $X \in|p_1^*(D_1)+ d_2C_0|$ be a curve in  $C \times \mathbb{P}^1$. Let $\phi:X \rightarrow \mathbb{P}^1$  be a morphism such that there exists $\psi:S \rightarrow \mathbb{P}^1$ such thath $\psi_{|_{X}}=\phi$. Then $\text{deg}(\phi) \geq \text{min}(d_2\text{gon}(C),d_1)$.
\end{lemma}
\begin{proof}
    Let $D$ be a fiber of $\psi$. Then $ D \sim p_1^*B + aC_0 $, with $a\in \mathbb{Z}$ and $B$ a divisor in $C$ of degree $b$. Numerically: $D \equiv bf + aC_0 $.  From $f \cdot D \geq 0$, $C_0 \cdot D \geq 0$, and $D^2=0$ one finds $a \geq 0$, $b \geq 0$ and $2ab=0$. Then we have two cases:
    \begin{itemize}
\item[(i)] $a=0$. In this case $ D \sim p_1^*B$. Then $\text{deg}(\phi)=\text{deg}(\psi_{|_{X}})=X \cdot D=d_2b \geq d_2\text{gon}(C)$, where the latter inequality follows from the observation that the restriction of $\psi$ to a fiber of $p_2$ gives  a morphism $C \rightarrow \mathbb{P}^1$ of degree greater or equal than $C_0 \cdot D=b$. And so $b \geq \text{gon}(C)$.
\item[(ii)] $b=0$. In this situation $D \sim aC_0$ and then $\text{deg}(\phi)=\text{deg}(\psi_{|_{X}})=aC_0\cdot X=ad_1\geq d_1$.
    \end{itemize}
    \end{proof}
\begin{proof}[Proof  of proposition \ref{goncurvesprod}]
Let $X \in |p_1^*(D_1)+ d_2C_0|$ be a curve in $C \times \mathbb{P}^1$ as before. Denote by $k$ the gonality of $X$ and let $\phi: X \rightarrow \mathbb{P}^1$ a morphism of degree $k$. If $\phi$ is extendable we conclude using lemma \ref{lemmino}. Then, let $\phi$ be not extendable. By contradiction suppose  $k < \text{min}(d_1,d_2\text{gon}(C_1))$. 
 By Serrano's theorem we get $X^2=2d_1d_2 \leq \frac{1}{2}(k+2)^2 < \frac{1}{2}(d_1+2)^2$. That cannot happen if $d_2 \geq \frac{d_1}{4}+1+\frac{1}{d_1}$. Finally observe that from  $k < \text{min}(d_1,d_2\text{gon}(C_1))$,  we get $(k+1)^2 \leq d_1d_2\text{gon}(C_1)$ and so, if  $C$ is hyperlelliptic, we get $(k+1)^2 \leq 2d_1d_2=X^2\leq \frac{1}{2}(k+2)^2 \implies k=1$ and $X \simeq \mathbb{P}^1$. 
\end{proof}
We finish proving a lemma which gives a criterion for a line bundle of the type $\omega_C-T_m+\alpha$ to be base point free/  very ample. We will use it in proposition \ref{surjimpds4} and  theorem \ref{propsurjgeneral23}.

Since we want this lemma to hold for any effective divisor $T_m$ of degree $m$, we have to suppose $m \leq g-3$. This condition in fact guarantees that $h^0(C,\omega_c-T_m + \alpha)\geq 2$. 
\begin{lemma}\label{lemmacliffandalph}
Let $C$ be a curve, $T_m$ an effective divisor of degree $m \leq g-3$ and $\alpha$ a (non trivial) $2-$torsion element. 
\begin{itemize}
\item[(a)]Suppose $\omega_C-T_m+\alpha$ is not base point free. Then
\begin{itemize}
\item[(i)] $h^0(C,T_m+\alpha)=0$ and  there exist a point $p$ such that $\text{dim}(|2(T_m+p)|) \geq 1$; or
\item[(ii)] $h^0(C,T_m+\alpha)\geq 1$  and  there exist a point $p$ such that $\text{dim}(|T_m+ \alpha +p)|) \geq 1$. 
\end{itemize}
\item [(b)] Suppose $\omega_C-T_m+\alpha$ is not very ample. Then  
\begin{itemize}
\item[(i)]there exist points $p$ and $q$ such that $h^0(C,T_m+\alpha+p)=0$  and $\text{dim}(|2(T_m+p+q)|) \geq 1$; or 
\item[(ii)] there exist points $p$ and $q$ such that $h^0(C,T_m+\alpha+p)\geq 1$  and  $\text{dim}(|T_m+ \alpha +p+q)|) \geq 1$. 
\end{itemize}
\end{itemize}
\end{lemma}
\begin{proof}
\
\\
\begin{itemize}
\item[(a)] Suppose $\omega_C - T_m + \alpha$ has  a base point $p$. Then, by Riemann-Roch, $h^0(T_m+p+\alpha)=h^0(T_m+\alpha)+1$. If $h^0(T_m+\alpha) \geq 1$ we conclude. If $h^0(T_m+\alpha)=0$, then  $h^0(T_m+\alpha+p)=1$. Then there exists an effective divisor $E$ 
such that $E \sim T_m+p+\alpha$. This gives
$2E \sim 2(T_m+p)$. 
Now observe that $h^0(2(T_m+p) \geq 2$ since  otherwise $2E=2(T_m+p)$  and hence $E=T_m+p$ which gives $\alpha=0$. This cannot be the case  because $\alpha$ is not trivial  by hypothesis.
\item[(b)] 
Suppose there exists two points $p$ and $q$ such that $q$ is a base point of  $\omega_C - T_m + \alpha-p$. Then, by Riemann-Roch, $h^0(T_m+p+q+\alpha)=h^0(T_m+p + \alpha)+1$. If $h^0(T_m+p+\alpha) \geq 1$ we conclude. If $h^0(T_m+p+\alpha)=0$, then  $h^0(T_m+\alpha+p+q)=1$. Then there exists an effective divisor $E$ 
such that $E \sim T_m+p+q+\alpha$.  Then
$2E \sim 2(T_m+p+q)$. As before, it follows $h^0(2(T_m+p+q))\geq 2.$ 
\end{itemize}
\end{proof}
%
%
%
\section{Surjecitivty for general curves}
\label{section5}
Let $\widetilde{M}_{g,d}$ be the Deligne-Mumford stack of smooth curves of genus $g$ with $d$ unordered points. Let $R_g$ be the stack of Prym curves with genus $g$. We consider the stack of Prym curves of genus $g$ with $d$ unordered points:
\begin{equation}
\label{rfdfdgd}
R_{g,d}:=R_g \times_{M_g} \widetilde{M}_{g,d}.
\end{equation}
We want to show that for the general point in $R_{g,d}$, under some assumptions on $g,d$, the gaussian maps $ \Phi_{C,\omega_C - T_{d},\omega_C-T_{d} + \alpha, }
    $ and $ \Phi_{\omega_C,\omega_C-T_{d} + \alpha}$ are surjective.
%
More precisely let us introduce the following sets:
\begin{equation}
\label{equadhfd}
S:=\{(g_1,d_1,d_2):  g_1 \geq 3, d_1 \geq 5, d_2 \geq 4, d_2(g_1-2)>d_1 \geq g_1+5, d_1 >d_2  \}.
\end{equation}
Fix $(g_1,d_1,d_2) \in S$ and set $g=(g_1-2)d_2+d_1(d_2-1)+1$. We are going to prove that for all  $ 0 \leq d \leq d_2$, if $(C,\alpha,T_d) \in R_{g,d}$ is a general element  the gaussian maps $ \Phi_{C,\omega_C - T_{d},\omega_C-T_{d} + \alpha, }
    $ and $ \Phi_{\omega_C,\omega_C-T_{d} + \alpha}$ are surjective.
\begin{notations}
\label{refdls}
In the following we will denote by $(C^*,\alpha^*, T_{d_{2}}^*)$ a point in $R_{g,d_2}$ constructed as in construction \ref{curvsrufapic} with $D_1$ general and taking $C_2=\mathbb{P}^1$, $C_1$ hyperelliptic, $(g_1,d_1,d_2)$ belonging to $S$ and $g=(g_1-2)d_2+d_1(d_2-1)+1$. We observe that the conditions   $g_1 \geq 3, d_2 \geq 4, d_2(g_1-2)>d_1 \geq g_1+5$ guarantee  that $C^*$ does exist  - by remark \ref{remajdjdjs} - and  the surjectivity of the aforementioned gaussian maps for the special point (see corollary \ref{corsurjfirstgauss}).  We  require  $C_1$ to be  hyperelliptic and $d_1 >d_2$ because in the proof  proposition \ref{surjimpds4}  we will need  $h^0(C^*,T_{d_2}^*)=1$ (we will use  proposition \ref{goncurvesprod}). 
%
\end{notations}
%
 \begin{proposition}
 \label{surjimpds4}
Let $(g_1,d_1,d_2)$ be in $S$ (\ref{equadhfd}), and $g=(g_1-2)d_2+d_1(d_2-1)+1$. Then the Gaussian maps
$$
\Phi_{\omega_{C},\omega_{C}-T_{d_2}+\alpha}:R(\omega_C,\omega_C-T_{d_2} + \alpha) \rightarrow H^0(C,3 \omega_C-T_{d_2}+ \alpha)
$$
and
$$
\Phi_{\omega_C-T_{d_2},\omega_C-T_{d_2}+\alpha}:R(\omega_C-T_{d_2},\omega_C-T_{d_2} + \alpha)  \rightarrow H^0(C,3 \omega_C-2T_{d_2}+ \alpha)
$$
are surjective for the general curve marked prym curve in $R_{g,d_2}$ .
 \end{proposition}
 \begin{proof}
We will  prove the result for $\Phi_{\omega_C-T_{d_2},\omega_C-T_{d_2}+\alpha}$. An identical proof gives the surjectivity of $\Phi_{\omega_C,\omega_C-T_{d_2}+\alpha}$. 
For the rest of the proof we denote $\Phi^i_{\omega_C-T_{d_2},\omega_C-T_{d_2}+\alpha}$, $i=0,1$ just by $\Phi^i$,  $i=0,1$. Let $X$ be the product $C_1 \times \mathbb{P}^1$  with  $g(C_1)=g_1 \geq 3$ and $C_1$ hyperelliptic. Let $(C^*,\alpha^*,T_{d_2}^*)$ be a marked Prym curve constructed as in \ref{refdls}.

First, we show that $h^0(C,\omega_C-T_{d_2})$,  $h^0(C,\omega_C-T_{d_2}+\alpha)$ and $h^0(C,2\omega_C-2T_{d_2}+\alpha)$ are locally constant in an neighborhood of $(C^*,\alpha^*,T_{d_2}^*)$.
For the latter, immediately follows from Riemann-Roch. So let's  focus on the other two. By Riemann-Roch it is equivalent to show that $h^0(C,T_{d_2}+\alpha)$ and $h^0(C,T_{d_2})$ are locally constant in a neighborhood of $(C^*,T_{d_2}^*,\alpha^*)$.  For the special point we have $h^0(C^*,T_{d_2}^*)=1$ since $d_1 >d_2$ by construction and gon$(C^*)>\text{min}(d_1, 2d_2)>d_2$ by proposition \ref{goncurvesprod}. Next we show that $h^0(C^*,T_{d_2}^*+\alpha^*)=0$. Consider:
$$
0 \rightarrow \mathcal{O}_X( p_1^*p_{1,1}+ p_1^*\alpha'(-C^*))\rightarrow \mathcal{O}_X(p_1^*p_{1,1}+ p_1^*\alpha') \rightarrow \mathcal{O}_{C^*}(p_1^*p_{1,1}+ p_1^*\alpha') \rightarrow 0.
$$
 By Künneth formula we have that
$H^1(X,\mathcal{O}_X( p_1^*p_{1,1}+ p_1^*\alpha'(-C^*))\simeq$
 \begin{align*}
H^0(C_1,\mathcal{O}_{C_1}(p_{1,1}+\alpha'(-D_1)) &\otimes H^1(C_2,\mathcal{O}_{C_2}(-D_2))
\
\\
&\oplus 
\
\\
H^1(C_1,\mathcal{O}_{C_1}(p_{1,1}+\alpha'(-D_1)) &\otimes H^0(C_2,\mathcal{O}_{C_2}(-D_2)).
 \end{align*}
and notice that $h^0$ terms are zero by the hypothesis on the degree of $D_i$.  Now observe that choosing $p_{1,1} \in \text{supp}(D_1)$ general in the construction, we can assume $h^0(X,\mathcal{O}_X(p_1^*p_{1,1}+ p_1^*\alpha'))=h^0(C_1,p_{1,1}+ \alpha')=0$. Therefore $h^0(C^*,T^*_{d_2}+\alpha)=h^0(C^*,\mathcal{O}_{C^*}(p_1^*p_{1,1}+ p_1^*\alpha'))=0$. 

Now observe that since $g=(g_1-2)d_2+d_1(d_2-1)+1\geq d_2$, if $(C,T_{d_2},\alpha)$ is a general point $R_{g,d_2}$, (in particular we can assume $T_{d_2}$ to consist of $d_2$ general points) $h^0(C,T_{d_2})=1$.  Analogously, since $g-1 \geq d_2, h^0(C,T_{d_2}+\alpha)=0$. Hence we are done: the dimensions of the spaces of global sections of the line bundles we are considering are locally constant in a neighborhood of the special point.  
By proposition \ref{surjmult}, $\Phi^0$ is surjective for the special point $(C^*,\alpha^*, T_{d_2}^*)$. Then $R(\omega_C-T_{d_2},\omega_C-T_{d_2} + \alpha)$ has constant rank in a neighborood of the special point. By Riemann-Roch, also  $h^0(C,3\omega_C-2T_{d_2}+\alpha)$ is locally constant. Since, by corollary \ref{corsurjfirstgauss} $\Phi$ is surjective for the special point, by semi-coninuity it is surjective in a neighborhood.
%
 For the Gaussian map $\Phi_{\omega_C-T_{d_2},\omega_C-T_{d_2}+\alpha}: R(\omega_C ,\omega_C-T_{d_2}+\alpha) \rightarrow H^0(C,3 \omega_C-2T_{d_2}+ \alpha)$ the proof is very similar.
\end{proof}
We observe that the previous result  requires $d_2 \geq 4$ and hence we don't still have a surjectivity result for a general Prym curve with $2 $ or $3$ marked points.  We overcome this problem in the final theorem (theorem \ref{propsurjgeneral23}).
    are  surjective.
\begin{proof}[Proof of theorem \ref{propsurjgeneral23}]
Let's deal with the first map. Let $(C, \alpha, T_{d})$ and $(C, \alpha, T_{d_2})$  be general  elements in $R_{g,d}$ and $R_{g,d_2}$ respectively, such that $T_d \subseteq T_{d_2}$. In particular we can suppose that the supports of $T_{d}$ and $T_{d_2}$ consists of distinct points and that gon$(C)=[\frac{g+3}{2}]$. An easy calculation shows that $[\frac{g+3}{2}] > 2(d_2+2)$, and so we have that $\omega_C-T_{d_2} +\alpha$ is very ample.  In fact, if $\omega_C-T_{d_2}+\alpha$ is not very ample, by lemma \ref{lemmacliffandalph} there exists a $g^1_{2(d_2+2)}$.  Observe that also $\omega_C-T_{d_2}$ and $\omega_C-T_{d}$ are very ample since otherwise the curve would admit a $g^1_{d_2+2}$ and a $g^1_{d+2}$ respectively. Moreover, observe that $h^1(\omega_C-T_{d})=h^1(\omega_C-T_{d_2})$ since $h^0(T_{d})=h^0(T_{d_2})=1$. Denote by $T_n$ the $n-$ distinct points such that $T_{d_2}=T_d+T_n$. Then we can apply proposition \ref{propositionfundamentalsurjection} with $L=\omega_C-T_{d}$, $n=d_2-d$  (then $L-T_n=\omega_C-T_{d_2}$) and $M=\omega_C-T_{d_2} + \alpha $ and obtain a surjective map 
$$
 \operatorname{coker}(\Phi_{\omega_C-T_{d_2},\omega_C-T_{d_2} + \alpha }) \rightarrow \operatorname{coker}(\Phi_{\omega_C-T_{d},\omega_C-T_{d_2} + \alpha }).
$$
Sincer $coker((\Phi_{\omega_C-T_{d_2},\omega_C-T_{d_2} + \alpha })=0$ by proposition \ref{surjimpds4}, we conclude that $\text{coker}(\Phi_{\omega_C-T_{d},\omega_C-T_{d_2} + \alpha })$ is zero. Now we use again proposition \ref{propositionfundamentalsurjection} with $L=\omega_C-T_{d} + \alpha$, $n=d_2-d$ and $M=\omega_C-T_{d}  $ (in particular $L-T_n=\omega_C-T_{d_2} + \alpha$). Notice that $h^1(L)=h^1(L-T_n)$ since $
h^0(T_d+\alpha)=h^0(T_{d_2}+\alpha)=0$, being $T_d$ and $T_{d_2}$ general points in $C$, and $g-1 \geq d_2 \geq d$. 
We then obtain a surjective map:
$$
 \operatorname{coker}(\Phi_{\omega_C-T_{d_2} + \alpha ,\omega_C-T_{d}  }) \rightarrow \operatorname{coker}(\Phi_{\omega_C-T_{d} + \alpha ,\omega_C-T_{d}  }), 
$$
and so we conclude that $\operatorname{coker}(\Phi_{\omega_C-T_{d} , \omega_C-T_{d} + \alpha} )=0$ for the general element. The proof for $$  \Phi_{C,\omega_C, \omega_C - T_d + \alpha}$$
is analogous. 
\end{proof}

\begin{example}
\label{exampledsds}
Observe that choosing $d_2=4$, $d_1=g_1+l+5$ with $g_1 \geq 3$ and  $0 \leq l+6 \leq 3g_1$, 
 all the conditions are satisfied and in this case $g=7g_1+3l+8$. Choosing $(g_1,l) \in \{(3+k,4),(4+k,2),(5+k,0),(3+k,5),(4+k,3),(5+k,1),(3+k,6), \ k \geq 0\}$,  we get all the genera greater or equal than $41$. Then, by theorem \ref{propsurjgeneral23}, for all $g \geq 41$ the Gaussian maps  with $2$,$3$ or $4$-marked points are surjective.
\end{example}
%
\begin{remark}
\label{remarksbounds}
    We expect our results regarding the surjectivity of \ref{prima} and   \ref{seconda} to be not sharp. In this remark, we compute the  expected  numerical range of degrees $d$ and genus $g$ such that one can aspect the surjectivity of the gaussian maps for the general element $(C,\alpha, T_d)$. 
    Denote by $\Phi^0$, $\Phi$ ($\Phi^{0'}$, $\Phi^{'}$) respectively $\Phi^0_{\omega_C,\omega_C -T_{d}+\alpha}$ and  $\Phi_{\omega_C,\omega_C -T_d+\alpha}$ ($\Phi^0_{\omega_C-T_d,\omega_C -T_{d}+\alpha}$ and  $\Phi_{\omega_C-T_d,\omega_C -T_d+\alpha}$),  and denote by  $R(g,d)$ ($R'(g,d)$) the kernel of $\Phi^0$ ($\Phi^{0'}$). We first observe that a necessary condition for surjectivity is $d\geq g-3$. Indeed, let $(C,\alpha,T_d)$ be a general element in $R_{g,d}$. Observe that $h^0(C,\omega_C+\alpha-T_d)=\max\{g-1-d,0\}$. Then, if $ d\geq g-1$, $R(g,d)=0$. If $d=g-2$, $h^0(C,\omega_C+\alpha-T_d)=1$ and $\Phi^0$ ($\Phi^{0'}$) is injective in both cases. Then suppose $d \leq g-3$. An easy calculation shows that in order to have the surjecitivity of $\ref{prima}$, we need $d \leq g-7-\frac{6}{g-2}+ \frac{\text{cork}(\Phi^0)}{g-2}$. In particular, one can aspect to have surjectivity of $\ref{prima}$ for every $g\geq 9$ and $d \leq g-8$. 

    Analogously, an easy calculation shows that in order to have the surjectivity of $\ref{seconda}$, we need $d \leq g-3 $, if $  g=4$ or $g=5$, and  $d \leq g-\frac{5}{2}-\sqrt{8g-7}+\text{cork}(\Phi^{0^{'}})$ if $g \geq 6$.
    
    We plan to return to the problem in the next future and improve the conditions on $g$ and $d$ to have the surjectivity of the Gaussian maps.

\end{remark}

	\end{document}